\documentclass[12pt, reqno]{amsart} 
\usepackage{amsmath, amsthm, amscd, amsfonts, amssymb, graphicx, color}
\usepackage[bookmarksnumbered, colorlinks, plainpages]{hyperref}
\hypersetup{colorlinks=true,linkcolor=red, anchorcolor=green, citecolor=cyan, urlcolor=red, filecolor=magenta, pdftoolbar=true}

\newtheorem{theorem}{Theorem}[section]
\newtheorem{lemma}[theorem]{Lemma}
\newtheorem{prop}[theorem]{Proposition}
\newtheorem{cor}[theorem]{Corollary}
\theoremstyle{definition}

\theoremstyle{remark}
\newtheorem{remark}[theorem]{\bf{Remark}}
\numberwithin{equation}{section}
\begin{document}

\title [Improved bounds for the numerical radius]    
{ {  Improved bounds for the numerical radius via polar decomposition of operators   }}

%Refinement of numerical radius inequalities via polar decomposition of operators 

\author[P. Bhunia]{Pintu Bhunia}

\address{Department of Mathematics, Indian Institute of Science, Bengaluru 560012, Karnataka, India}
\email{pintubhunia5206@gmail.com; pintubhunia@iisc.ac.in }

%\thanks will become a 1st page footnote.
%\thanks{We would like to thank the referee for his/her extremely fruitful suggestions. First author  would like to thank UGC, Govt. of India for the financial support in the form of SRF. Prof. Kallol Paul would like to thank RUSA 2.0, Jadavpur University for the partial support.}

\thanks{Dr. Pintu Bhunia would like to sincerely acknowledge Professor Kallol Paul for his valuable comments on this article.
 Dr. Bhunia also would like to thank SERB, Govt. of India for the financial support in the form of National Post Doctoral Fellowship (N-PDF, File No. PDF/2022/000325) under the mentorship of Professor Apoorva Khare}

\thanks{}
%    General info


\date{}
\maketitle

\begin{abstract}
Using the polar decomposition of a bounded linear operator $A$ defined on a complex Hilbert space, we obtain several numerical radius inequalities of the operator $A$, which generalize and improve the earlier related ones. Among other bounds, we show that if $w(A)$ is the numerical radius of $A$, then 
\begin{eqnarray*}
	w(A) &\leq&  \frac12 \|A\|^{1/2} \left\| |A|^{t} + |A^*|^{1-t}  \right \|, 
\end{eqnarray*}
	for all $t\in [0,1].$
Also, we obtain some upper bounds for the numerical radius involving the spectral radius and the Aluthge transform of operators. It is shown that 
	\begin{eqnarray*}
	w(A) &\leq&  \|A\|^{1/2} \left( \frac12 \left \| \frac{ |A|+|A^*|}2 \right\|  +\frac12 \left\| \widetilde{A}\right \| \right)^{1/2},
\end{eqnarray*}
where $\widetilde{A}= |A|^{1/2}U|A|^{1/2} $ is the Aluthge transform of $A$ and $A=U|A|$ is the polar decomposition of $A$.
Other related results are also provided.
\end{abstract}

\section{\textbf{Introduction}}
\noindent
Let $\mathcal{B}(\mathcal{H})$ denote the $C^{*}$-algebra of all bounded linear operators on a complex Hilbert space $\mathcal{H}$, with inner product $\langle \cdot, \cdot \rangle$ and the corresponding norm $\|\cdot\|.$  For $A \in \mathcal{B}(\mathcal{H}),$ let $|A|=(A^*A)^{1/2}$ and $|A^*|=(AA^*)^{1/2}$, where $A^*$ is the adjoint of $A.$ For $t\in [0,1],$ the $t$-Aluthge transform (generalized Aluthge transform) of $A\in \mathcal{B}(\mathcal{H})$ is  $\widetilde{A_t}=|A|^{t}U|A|^{1-t},$
where $A=U|A|$ is the polar decomposition of $A$ and $U$ is the partial isometry. In particular, for $t=\frac12$, $\widetilde{A}=\widetilde{A_{\frac12}}=|A|^{1/2}U|A|^{1/2}$ is the Aluthge transform of $A.$

\smallskip
\noindent Let $\|A\|$, $w(A)$ and $r(A)$ denote the operator norm, the numerical radius and the spectral radius of $A$, respectively.
The numerical radius of $A$ is defined as
$$ w(A)=\sup \{|\langle Ax,x \rangle| : x\in \mathcal{H}, \|x\|=1  \}$$
and it is the radius of the smallest disc with center at origin that contains the numerical range. Note that the numerical range $W(A)$ is defined as $W(A)= \{\langle Ax,x \rangle : x\in \mathcal{H}, \|x\|=1  \}.$
It is well known that the numerical radius $w(\cdot):\mathcal{B}(\mathcal{H}) \to \mathbb{R} $ defines a norm on  $\mathcal{B}(\mathcal{H})$ and is equivalent to the operator norm. For  every $A \in  \mathcal{B}(\mathcal{H}),$
\begin{eqnarray}\label{eqv1}
	\frac12 \|A\| \leq  w(A) \leq \|A\|,
\end{eqnarray}
holds.
The spectral radius of $A$ is defined as 
$r(A)=\sup \left\{ |\lambda| : \lambda \in \sigma(A)\right\},$
where $\sigma(A)$ is the spectrum of $A.$ Since, the spectrum $\sigma(A)$  is contained in the closure of the numerical range  (i.e., $\sigma(A) \subset \overline{W(A)}$),  we have 
$$ r(A)\leq w(A).$$
Therefore, for every $A\in \mathcal{B}(\mathcal{H})$, $r(A)\leq w(A)\leq \|A\|$ holds, and  $r(A)=w(A)=\|A\|$, when $A \in \mathcal{B}(\mathcal{H})$ is a normal operator. Also, note that $r(A)=r(\widetilde{A})$, $w(\widetilde{A})\leq w(A)$ and $\|\widetilde{A}\|\leq \|A^2\|^{1/2}\leq \|A\|.$
For more details about the numerical range, the numerical radius and related inequalities, the readers can follow the books \cite{book, book2}.
Various refinements of the numerical radius bounds in \eqref{eqv1} have been studied over the years. Kittaneh in \cite[2003]{Kittaneh_2003} and \cite[2005]{Kittaneh_STD_2005}, respectively, developed the following bounds
\begin{eqnarray}\label{kit03}
	w(A) &\leq& \frac12 \|A\|+ \|A^2\|^{1/2}
\end{eqnarray}
and 
\begin{eqnarray}\label{kit05}
	w^2(A) &\leq&  \frac12 \left\||A|^2+|A^*|^2 \right\|.
\end{eqnarray}
The bounds in \eqref{kit03} and \eqref{kit05} improve the same in \eqref{eqv1}.
 Dragomir in \cite[2008]{D08} proved that
\begin{eqnarray}\label{dra08}
	w^2(A) &\leq& \frac12 \|A\|^2+ \frac12w(A^2).
\end{eqnarray}
Clearly, the bound in \eqref{dra08} improves the same in \eqref{eqv1}. After that, Abu-Omar and Kittaneh in \cite[2015]{Abu_RMJM_2015} developed 
\begin{eqnarray}\label{abu15}
	w^2(A) &\leq&  \frac14 \left\||A|^2+|A^*|^2 \right\|+ \frac12 w(A^2),
\end{eqnarray}
which improves both the bounds in \eqref{kit03},\eqref{kit05} and \eqref{dra08}.
Further, Bhunia and Paul in \cite[2021]{Bhunia_BSM_2021} proved that
\begin{eqnarray}\label{bhu21}
	w^2(A) &\leq&  \frac14 \left\||A|^2+|A^*|^2 \right\|+ \frac12 w(|A||A|^*)
\end{eqnarray}
and 
\begin{eqnarray}\label{bhu_21}
	w(A) &\leq&  \frac1{\sqrt{2}}  w(|A|+i|A|^*).
\end{eqnarray}
The bound in \eqref{bhu21}, is incomparable with the bound in \eqref{abu15}, refines both the bounds in \eqref{kit03} and \eqref{kit05}.
The bound in \eqref{bhu_21} refines the bound \eqref{kit05}.
The same authors in \cite[2021]{Bhunia_AM_2021} obtained an improvement of \eqref{kit03} by using the spectral radius, namely,
\begin{eqnarray}\label{bhu211}
	w(A) &\leq&  \frac12 \|A\|+ \frac12 r^{1/2}(|A||A|^*).
\end{eqnarray}
Further, Bhunia \cite[2023]{Bhu23} obtained an improvement of the second inequality in \eqref{eqv1}, namely,
\begin{eqnarray}\label{bhunia_21}
	w(A) &\leq&  \|A\|^{1/2}   \big \|\alpha |A|+(1-\alpha)|A^*|  \big\|^{1/2}, 
\end{eqnarray}
for all $\alpha\in [0,1]$. Recently, Kittaneh, Moradi and Sababheh \cite[2023]{Kittaneh_LAMA_2023} also
developed the following nice improvement of the second inequality in \eqref{eqv1}:
\begin{eqnarray}\label{kit23}
	w(A) &\leq&  \frac12 \|A\|^{1/2}   \left \| |A|^{1/2}+|A^*|^{1/2}  \right\|.
\end{eqnarray}
Some generalizations of the inequalities in \eqref{kit05}, \eqref{abu15} \eqref{bhu21} and  \eqref{bhu_21} and the other improvements also studied, see \cite{Abu_STU_2013,Bag_MIA_2020, Bhunia_ASM_2022,  Bhunia_LAA_2021,Bhunia_RIM_2021,Bhunia_BSM_2021,Bhunia_LAA_2019, El_STU_2007, Yam_STU_2007}.

\smallskip

In this paper, we obtain various numerical radius inequalities of bounded linear operators, which generalize and improve on the bounds in \eqref{eqv1}, \eqref{kit03}, \eqref{kit05}, \eqref{bhu21}, \eqref{bhu_21}, \eqref{bhu211}, \eqref{bhunia_21} and \eqref{kit23}. Other bounds are also developed which refine the existing ones.

\section{\textbf{Main Results}}

We begin our study with the following known lemmas. First lemma is  known as McCarthy inequality.

 \begin{lemma} \cite{MC}\label{lem1}
 	Let $A\in \mathcal{B}(\mathcal{H})$ be positive, and let  $x\in \mathcal{H}$ with $\|x\|=1$.  Then 
 	$$ \langle Ax,x\rangle^p \leq \langle A^px,x\rangle, $$
 	 for all $p\geq 1.$
 \end{lemma}

Second lemma involves $2\times 2$ positive operator matrix.

\begin{lemma}\cite[Lemma 1]{Kit88}\label{lem2}
	Let $A,B,C\in \mathcal{B}(\mathcal{H})$, where $A$ and $B$ are positive.
	Then the operator matrix $\begin{bmatrix}
		A&C^*\\
		C&B
	\end{bmatrix}\in \mathcal{B}(\mathcal{H}\oplus \mathcal{H})$ is positive if and only if $$|\langle Cx,y\rangle|^2\leq \langle Ax,x\rangle\langle By,y\rangle.$$
\end{lemma}
 
 Third lemma is named as Buzano's inequality.
 
 \begin{lemma}\cite{Buzano}\label{lem3}
 	Let $x,y,z\in \mathcal{H}$, where  $\|z\|=1$. Then
 	$$ |\langle x,z\rangle\langle z,y\rangle| \leq \frac{\|x\|\|y\|+|\langle x,y\rangle|}{2}.$$
 \end{lemma}

  By using the above lemmas we first prove the following proposition.
   %(Though the proof is easy, for reader convenience here we sketch a proof.)

\begin{prop}\label{prop1}
	Let $A,B,C\in \mathcal{B}(\mathcal{H})$, where $A$ and $B$ are positive.
	If $\begin{bmatrix}
		A&C^*\\
		C&B
	\end{bmatrix}\in \mathcal{B}(\mathcal{H}\oplus \mathcal{H})$ is positive, then the following bounds hold: \\
(i) $w^2(C) \leq \frac12 \left \| A^2+B^2\right\|.$\\
(ii)  $w^2(C) \leq \frac12 \left \| A \right\| \left\| B\right\|+ \frac12 w(AB).$\\
(iii)  $w^2(C) \leq \frac14 \left \| A^2+B^2\right\|+ \frac12 w(AB).$\\
(iv)  $w^2(C) \leq \left \| \alpha A+ (1-\alpha)B\right\| \|A\|^{1-\alpha}\|B\|^{\alpha}$, for all $\alpha\in [0,1] $.

\end{prop}

\begin{proof}
	Take $x\in \mathcal{H}$ with $\|x\|=1.$ \\
	(i) From Lemma \ref{lem2}, we have
	\begin{eqnarray*}
		|\langle Cx,x\rangle|^2 &\leq& \langle Ax,x\rangle\langle Bx,x\rangle\\
		&\leq& \frac12 \left(\langle Ax,x\rangle^2+ \langle Bx,x\rangle^2\right)\\
		&\leq& \frac12 \left(\langle A^2x,x\rangle+ \langle B^2x,x\rangle\right) \text{ (by Lemma \ref{lem1})}\\
		&\leq& \frac12 \left\| A^2+B^2\right\|.
	\end{eqnarray*}
	This implies, $w^2(C)\leq \frac12 \left\| A^2+B^2\right\|.$\\
	(ii) From Lemma \ref{lem2}, we have
	\begin{eqnarray*}
		|\langle Cx,x\rangle|^2 &\leq& \langle Ax,x\rangle\langle x,Bx\rangle\\
		&\leq& \frac12 \left( \|Ax\|\|Bx\|+ |\langle Ax,Bx\rangle|\right)  \text{ (by Lemma \ref{lem3})}\\
		&\leq& \frac12 \left(\|A\|\|B\|+ w(AB)\right).
	\end{eqnarray*}
	This gives, $w^2(C) \leq  \frac12 \|A\|\|B\|+ \frac12 w(AB).$\\
	(iii) From Lemma \ref{lem2}, we have
	\begin{eqnarray*}
		|\langle Cx,x\rangle|^2 &\leq& \langle Ax,x\rangle\langle x,Bx\rangle\\
		&\leq& \frac12 \left( \|Ax\|\|Bx\|+ |\langle Ax,Bx\rangle|\right)  \text{ (by Lemma \ref{lem3})}\\
		&\leq& \frac12 \left( \frac{\|Ax\|^2+\|Bx\|^2}2+ |\langle Ax,Bx\rangle|\right)\\
		&\leq& \frac14 \left \| A^2+B^2\right\|+ \frac12 w(AB).
	\end{eqnarray*}
This implies, $w^2(C) \leq \frac14 \left \| A^2+B^2\right\|+ \frac12 w(AB).$\\
(iv) From Lemma \ref{lem2}, we have
\begin{eqnarray*}
	|\langle Cx,x\rangle|^2 &\leq& \langle Ax,x\rangle\langle Bx,x\rangle\\
	&=& \langle Ax,x\rangle^{\alpha}\langle Bx,x\rangle^{1-\alpha}\langle Ax,x\rangle^{1-\alpha}\langle Bx,x\rangle^{\alpha} \\
	&\leq& \left( \alpha \langle Ax,x\rangle +(1-\alpha)\langle Bx,x\rangle \right) \langle Ax,x\rangle^{1-\alpha}\langle Bx,x\rangle^{\alpha} \\
	&\leq&  \left \| \alpha A+ (1-\alpha)B\right\| \|A\|^{1-\alpha}\|B\|^{\alpha}.
\end{eqnarray*}
This implies, $w^2(C) \leq \left \| \alpha A+ (1-\alpha)B\right\| \|A\|^{1-\alpha}\|B\|^{\alpha},$ as desired.
\end{proof}

By applying Proposition \ref{prop1}, we prove the following lemma.

\begin{lemma}\label{lem4}
	Let $B,C\in \mathcal{B}(\mathcal{H})$. Then $w(BC)$ satisfies the following bounds:\\
	(i) $w^2(BC) \leq \frac12 \left\||B^*|^4+|C|^4 \right\|.$\\
	(ii) $w^2(BC) \leq \frac12 \|B\|^2\|C\|^2+ \frac12 w\left(B(CB)^*C\right)$.\\
	(iii)  $w^2(BC) \leq \frac14 \left\| |B^*|^4+ |C|^4\right\|+ \frac12 w\left(B(CB)^*C\right)$.\\
	(iv)  $w^2(BC) \leq  \left\| \alpha |B^*|^2+ (1-\alpha)|C|^2\right\| \|B\|^{2(1-\alpha)}\|C\|^{2\alpha},$ for all $\alpha\in [0,1].$
\end{lemma}
\begin{proof}
	Following Lemma \ref{lem2}, it is easy to observe that the operator  matrix $\begin{bmatrix}
	BB^*	& BC \\
	C^*B^*	& C^*C
	\end{bmatrix}\in \mathcal{B}(\mathcal{H} \oplus \mathcal{H})$ is positive. Using this positive operator matrix in Proposition \ref{prop1}, we obtain the desired upper bounds of $w(BC)$.
\end{proof}

Now, we are in a position to obtain our first aim result.

\begin{theorem}\label{th1}
	If $A\in \mathcal{B}(\mathcal{H})$, then
	\begin{eqnarray*}
		w^2(A) \leq  \frac{1}{2} \left\| A   \right \|^2 +\frac12 w\left(|A|^{2t}|A^*|^{2(1-t)} \right),
	\end{eqnarray*}
 for all  $t\in [0,1]$.
 In particular, for $t=\frac12$
\begin{eqnarray}\label{par1}
	w^2(A) \leq  \frac{1}{2} \left\| A   \right \|^2 +\frac12 w\left(|A||A^*|\right).
\end{eqnarray}
\end{theorem}
\begin{proof}
Let $A=U|A|$ be the polar decomposition of $A$. By taking $B=U|A|^{1-t}$ and $C=|A|^t$ in (ii) of Lemma \ref{lem4}, we obtain  
\begin{eqnarray}\label{p1}
	w^2(A) \leq  \frac{1}{2} \left\| A   \right \|^2 +\frac12 w\left(U |A|^{1-t} \widetilde{A_t}^*|A|^{t} \right),
\end{eqnarray}
where $ \widetilde{A_t}=|A|^tU|A|^{1-t}$ is the $t$-Aluthge transform of $A.$ Now, it is easy to see that $U |A|^{1-t} \widetilde{A_t}^*|A|^{t} = |A^*|^{2(1-t)} |A|^{2t}.$ This completes the proof.
\end{proof}

\begin{remark}\label{rem3} 
	(i) Since $w(|A||A^*|)\leq \|A^2\|,$  
	\begin{eqnarray*}
		w^2(A)&\leq& \frac{1}{2} \left\| A   \right \|^2 +\frac12 w\left(|A||A^*|\right)\\
		&\leq& \frac{1}{2} \left\| A   \right \|^2 +\frac12 \|A^2\|\\
		&\leq& \|A\|^2.
	\end{eqnarray*}
 Therefore, the bound in \eqref{par1} refines the second bound in \eqref{eqv1}.\\
	(ii) It follows from \eqref{p1} that
	\begin{eqnarray}\label{p2}
		w^2(A) \leq  \frac{1}{2} \left\| A   \right \|^2 +\frac12 \left\| |A|^{t} \widetilde{A_t}|A|^{1-t} \right\|, 
	\end{eqnarray}
for all $t\in [0,1]$.
In particular, for $t=\frac12$
\begin{eqnarray}\label{p3}
	w^2(A) \leq  \frac{1}{2} \left\| A   \right \|^2 +\frac12 \left\| |A|^{1/2} \widetilde{A}|A|^{1/2} \right\|,
\end{eqnarray}
where $\widetilde{A}= \widetilde{A_{\frac12}}=|A|^{1/2}U|A|^{1/2}$ is the Aluthge transform of $A.$
\end{remark}

Next theorem reads as:

\begin{theorem}\label{th2}
	If $A\in \mathcal{B}(\mathcal{H})$, then
	\begin{eqnarray*}
		w^2(A) \leq  \frac{1}{4} \left\| |A|^{4t}+ |A^*|^{4(1-t)}   \right \| +\frac12 w\left(|A|^{2t}|A^*|^{2(1-t)} \right), 
	\end{eqnarray*}
	for all $t\in [0,1]$.
	In particular, for $t=\frac12$
	\begin{eqnarray}\label{par2}
		w^2(A) \leq  \frac{1}{4} \left\| |A|^{2}+ |A^*|^{2}   \right \| +\frac12 w\left(|A||A^*|\right).
	\end{eqnarray}
\end{theorem}
\begin{proof}
	Let $A=U|A|$ be the polar decomposition of $A$. By taking $B=U|A|^{1-t}$ and $C=|A|^t$ in (iii) of Lemma \ref{lem4} and using similar arguments as in Theorem \ref{th1}, we obtain the desired results.
\end{proof}

\begin{remark}\label{rem4}
	(i) The bound in \eqref{par2} was also developed in \cite[Theorem 2.5]{Bhunia_BSM_2021} using different technique. \\
	%Note that the bound \eqref{par2} is sharper than the same in \eqref{par1}.\\
	(ii) Using  similar arguments as \eqref{p2} and \eqref{p3}, we can obtain 
	\begin{eqnarray}\label{p4}
		w^2(A) \leq  \frac{1}{4} \left\| |A|^{4t}+ |A^*|^{4(1-t)}   \right \| +\frac12 \left\| |A|^{t} \widetilde{A_t}|A|^{1-t} \right\|, 
	\end{eqnarray}
	for all $t\in [0,1]$.
	In particular, for $t=\frac12$
	\begin{eqnarray}\label{p5}
		w^2(A) \leq  \frac{1}{4} \left\| |A|^{2}+ |A^*|^{2}   \right \|  +\frac12 \left\| |A|^{1/2} \widetilde{A}|A|^{1/2} \right\|.
	\end{eqnarray}
\end{remark}

Next result reads as follows:

\begin{theorem}\label{th3}
	If $A\in \mathcal{B}(\mathcal{H})$, then
	\begin{eqnarray*}
		w(A) \leq  \left\| \alpha |A|^{2(1-t)}+ (1-\alpha)|A^*|^{2t}   \right \|^{1/2} \|A\|^{(1-\alpha)(1-t)+\alpha t}, 
	\end{eqnarray*}
\text{ for all $\alpha,t\in [0,1]$.}
In particular, for $t=\frac12$
\begin{eqnarray}\label{par3}
	w(A) \leq   \left\| \alpha |A|^{}+ (1-\alpha)|A^*|^{}   \right \|^{1/2} \|A\|^{1/2}, 
\end{eqnarray}
\text{ for all $\alpha\in [0,1]$.}
Also, in particular, for $\alpha=\frac12$
\begin{eqnarray}\label{par3_1}
	w(A) \leq   \left\|  \frac{|A|^{2t}+ |A^*|^{2(1-t)}}{2}   \right \|^{1/2} \|A\|^{1/2},
\end{eqnarray}
 \text{ for all $t\in [0,1]$.}
\end{theorem}

\begin{proof}
	Let $A=U|A|$ be the polar decomposition of $A$. The proof follows from  (iv) of Lemma \ref{lem4} by taking $B=U|A|^{1-t}$ and $C=|A|^t$.
\end{proof}

Note that, the bound in \eqref{par3} was also obtained in \cite[Theorem 2.8]{Bhu23} using different technique. 
For our next result we need the following lemma.

\begin{lemma}\label{lem5}\cite[Theorem 2.5]{Bhunia_LAMA_2022}
	Let $B,C\in \mathcal{B}(\mathcal{H})$ be such that $|B|C=C^*|B|$. If $f,g: [0,\infty) \to [0,\infty)$  are continuous functions with $f(\lambda)g(\lambda)=\lambda$, for all $\lambda\geq 0,$ then
	\begin{eqnarray*}
		w^p(BC) &\leq&  r^p(C) w\left( \begin{bmatrix}
			0& f^{2p}(|B|)\\
			g^{2p}(|B^*|) & 0
		\end{bmatrix}\right)=\frac12  r^p(C) \left\|f^{2p}(|B|)+g^{2p}(|B^*|)\right\|,
	\end{eqnarray*}
 for all $p\geq 1.$
\end{lemma}

Using the above lemma we obtain the following bound.

\begin{theorem}\label{th4}
	Let $f,g: [0,\infty) \to [0,\infty)$  be continuous functions with $f(\lambda)g(\lambda)=\lambda$, for all $\lambda\geq 0.$
	If $A\in \mathcal{B}(\mathcal{H})$, then
	\begin{eqnarray*}
		w^p(A) \leq  \frac12 \|T\|^{pt} \left\| f^{2p}\left(|A|^{1-t}\right) + g^{2p}\left(|A^*|^{1-t}\right)  \right \|,
	\end{eqnarray*}
for all $p\geq 1,$ and for all $t\in [0,1]$. In particular, for $p=1$
\begin{eqnarray}\label{par4}
	w(A) \leq  \frac12 \|A\|^{t} \left\| f^{2}\left(|A|^{1-t}\right) + g^{2}\left(|A^*|^{1-t}\right)  \right \|,
\end{eqnarray}
for all $t\in [0,1]$.
\end{theorem}

\begin{proof}
Let $A=U|A|$ be the polar decomposition of $A$.  By taking $B=U|A|^{1-t}$ and $C=|A|^t$ in Lemma \ref{lem5}, we get
\begin{eqnarray*}
	w^p(A) &\leq & \frac12 r^p(|A|^t) \left\| f^{2p}\left(|A|^{1-t}\right) + g^{2p}\left(|A^*|^{1-t}\right)  \right \|.
\end{eqnarray*}
Since $r(|A|^t)=\||A|^t\|=\|A\|^t$, we obtain the desired results.
\end{proof}

Considering $f(\lambda)=\lambda^{\alpha}$ and $  g(\lambda)=\lambda^{1-\alpha}$, $0\leq \alpha \leq 1$, in \eqref{par4} we obtain the following corollary.

\begin{cor}\label{cor1}
	If $A\in \mathcal{B}(\mathcal{H}),$ then
	\begin{eqnarray*}
		w(A) \leq  \frac12 \|A\|^{t} \left\| |A|^{2\alpha(1-t)} + |A^*|^{2(1-\alpha)(1-t)}  \right \|,
	\end{eqnarray*}
	for all $\alpha,t\in [0,1]$. In particular, for $t=\frac12$
	\begin{eqnarray}\label{par5}
		w(A) \leq  \frac12 \|A\|^{1/2} \left\| |A|^{\alpha} + |A^*|^{1-\alpha}  \right \|,
	\end{eqnarray}
for all $\alpha\in [0,1]$.
\end{cor}
 
 \begin{remark}
 (i) Let $A\in \mathcal{B}(\mathcal{H}).$	
 In particular, considering $\alpha=\frac12$ in \eqref{par5} we obtain  the following bound
 	\begin{eqnarray}\label{lama}
 		w(A) \leq  \frac12 \|A\|^{1/2} \left\| |A|^{1/2} + |A^*|^{1/2}  \right \|,
 	\end{eqnarray}
 	which was recently proved by Kittaneh et al. \cite{Kittaneh_LAMA_2023}. \\
 	(ii) It follows from the bound in \eqref{par5} that 
 	\begin{eqnarray}\label{lamai}
 		w(A)\leq \frac12 \|A\|^{1/2} \min_{\alpha\in [0,1]}\left\| |A|^{\alpha} + |A^*|^{1-\alpha}  \right \|,
 	\end{eqnarray}
 	for every \text{$A\in \mathcal{B}(\mathcal{H})$}. Clearly, the bound in \eqref{lamai} is sharper than that of the bound in \eqref{lama}.
 	Considering $A=\begin{bmatrix}
 		0&2&0\\
 		0&0&3\\
 		0&0&0
 	\end{bmatrix} \oplus \begin{bmatrix}
 		1
 	\end{bmatrix}$ (defined on $\mathbb{C}^3\oplus \mathbb{C}$), we have, $\left\| |A|^{\alpha} + |A^*|^{1-\alpha}  \right \|=\max \left\{2^{1-\alpha}, 2^{\alpha}+3^{1-\alpha}, 3^{\alpha}, 2   \right\}.$
 	Clearly, $$\left\| |A|^{1/2} + |A^*|^{1/2}  \right \|=\sqrt{2}+\sqrt{3}\approxeq 3.14626436994 $$ and $$\left\| |A|^{\alpha_0} + |A^*|^{1-\alpha_0}  \right \|=2^{\alpha_0}+3^{1-\alpha_0}\approxeq 2.98118458519, \text{ where $\alpha_0=\frac{87}{100}$}.$$
 	Hence, 
 	$  \min_{\alpha\in [0,1]}\left\| |A|^{\alpha} + |A^*|^{1-\alpha}  \right \|<  \left\| |A|^{1/2} + |A^*|^{1/2}  \right \| .$
 	\smallskip
 	This implies that the bound in \eqref{lamai} is a non-trivial refinement of the bound in \eqref{lama}. 
\end{remark}
 Next, we need the following lemma.
 \begin{lemma}\label{lem6}\cite[Corollary 2.7]{Bhunia_LAMA_2022}
 	Let $B,C\in
 	 \mathcal{B}(\mathcal{H})$ be such that $|B|C=C^*|B|$. Then
 	$$ w(BC) \leq \frac14 \left(\left\| |B|^2+|B^*|^2 \right\| + 2\|B^2\|   \right)^{1/2}  \left(\left\| |C|^2+|C^*|^2 \right\| + 2\|C^2\|   \right)^{1/2}.$$
 \end{lemma}

Now, we obtain an upper bound of $w(A)$ using the Aluthge transform of $A.$

%and let $A=U|A|$ be the polar decomposition of $A$. If $\widetilde{A}= |A|^{1/2}U|A|^{1/2} $ is the Aluthge transform of $A$
 
 \begin{theorem}\label{th5}
 	If $A\in \mathcal{B}(\mathcal{H})$, then
 	\begin{eqnarray*}
 		w(A) &\leq&  \|A\|^{1/2} \left( \frac12 \left \| \frac{ |A|+|A^*|}2 \right\|  +\frac12 \left\| \widetilde{A}\right \| \right)^{1/2},
 	\end{eqnarray*}
 where $\widetilde{A}= |A|^{1/2}U|A|^{1/2} $ is the Aluthge transform of $A$ and $A=U|A|$ is the polar decomposition of $A$.
 \end{theorem}

\begin{proof}
	 By taking $B=U|A|^{1/2}$ and $C=|A|^{1/2}$ in Lemma \ref{lem6}, we obtain the desired first inequality. The second inequality follows from the first inequality.
\end{proof}

\begin{remark}
(i) Clearly, 
$$ \|A\|^{1/2} \left( \frac12 \left \| \frac{ |A|+|A^*|}2 \right\|  +\frac12 \left\| \widetilde{A}\right \| \right)^{1/2} \leq \|A\|^{1/2} \left( \frac12  \|  A \|  + \frac12 \left\| \widetilde{A} \right\| \right)^{1/2}.$$
 Since  $\| \widetilde{A} \|\leq \|A^2\|^{1/2}$,  
\begin{eqnarray*}
	\|A\|^{1/2} \left( \frac12  \|  A \|  + \frac12 \| \widetilde{A} \| \right)^{1/2}
	\leq  \|A\|^{1/2} \left( \frac12 \|A\|  +\frac12  \| {A^2} \|^{1/2} \right)^{1/2} \leq \|A\|.
\end{eqnarray*}
Therefore, the bound obtained in Theorem \ref{th5} refines the second bound in \eqref{eqv1}.\\
(ii) Following \cite[Corollary 2]{Kittaneh_JFA_1997}, we have $\left\|  |A|+|A^*| \right\|\leq \|A\|+ \left\||A|^{1/2}|A^*|^{1/2}\right\|.$
Also, it is easy to observe that $\left\||A|^{1/2}|A^*|^{1/2}\right\|=r^{1/2}\left(|A||A^*|\right).$ Therefore, from Theorem \ref{th5}, we derive that
	\begin{eqnarray*}
	w(A) &\leq&  \|A\|^{1/2} \left( \frac12 \left \| \frac{ |A|+|A^*|}2 \right\|  +\frac12 \left\| \widetilde{A}\right \| \right)^{1/2}\\
	&\leq&  \|A\|^{1/2} \left( \frac12 \left(\frac12  \|  A \|+ \frac12 r^{1/2}\left(|A||A^*|\right) \right) + \frac12 \| \widetilde{A} \| \right)^{1/2}\\
	&\leq&  \|A\|^{1/2} \left( \frac12\left(\frac12  \|  A \|+ \frac12 w^{1/2}\left(|A||A^*|\right)\right)  + \frac12 \| \widetilde{A} \| \right)^{1/2}\\
	&\leq&  \|A\|^{1/2} \left( \frac12  \left(\frac12 \|  A \|+ \frac12 \left\|A^2\right\|^{1/2}\right)  + \frac12 \| \widetilde{A} \| \right)^{1/2}.
\end{eqnarray*}
\label{rem2}\end{remark}

 Based on the inequalities in (ii) of Remark \ref{rem2} and the first inequality in \eqref{eqv1}, we obtain the following proposition.
 
 \begin{prop}\label{prop2}
 	Let $A\in \mathcal{B}(\mathcal{H})$. If $A^2=0$, then 
 	 $$w(A)=\frac12 {\|A\|} \, \text{ and } \,
 	 \left \| { |A|+|A^*|} \right\|=\|A\|.$$
 	%(c) $|A|^{1/2}|A^*|^{1/2}=0.$
 \end{prop}
 %(i.e., Proposition \ref{prop2})
 The converse of the above proposition may not hold. For example, considering $A=\begin{bmatrix}
 	0&2\\
 	0&0
\end{bmatrix} \oplus \begin{bmatrix}
 1
\end{bmatrix}$ (defined on $\mathbb{C}^2\oplus \mathbb{C}$), we see that $w(A)=1=\frac12\|A\|$ and $ \left \| { |A|+|A^*|} \right\|=2=\|A\|$, but $A^2\neq0.$
%{\color{red} Does the converse of Proposition \ref{prop2} hold?}
\smallskip
To develop our next result we need the following lemma.
\begin{lemma}\label{lem7}\cite[Corollary 2.13]{Bhunia_AM_2021}
	Let $B,C\in \mathcal{B}(\mathcal{H})$ be such that $|B|C=C^*|B|$. Then
	$$ w(BC) \leq \frac12r(C)\left(\|B\|+ r^{1/2}(|B||B^*|)\right) .$$
\end{lemma}

Now, we prove the following theorem.

\begin{theorem}\label{th6}
	If $A\in \mathcal{B}(\mathcal{H})$, then
	$$ w(A) \leq \frac12 \|A\|+ \frac12 \|A\|^t r^{1/2}\left(|A|^{1-t} |A^*|^{1-t} \right),$$
	for all $t\in [0,1]$. In particular, for $t=\frac12$
	\begin{eqnarray}\label{par6}
		 w(A) \leq \frac12 \|A\|+ \frac12 \|A\|^{1/2} r^{1/2}\left(|A|^{1/2} |A^*|^{1/2} \right).
	\end{eqnarray}
\end{theorem}
\begin{proof}
Let $A=U|A|$ be the polar decomposition of $A$.	By putting $B=U|A|^{1-t}$ and $C=|A|^{t}$ in Lemma \ref{lem7}, we get the desired results.
\end{proof}

\begin{remark}
	Let $A\in \mathcal{B}(\mathcal{H}).$ 
	 For $0\leq t\leq 1$, we see that
	\begin{eqnarray*}
		r^{1/2}\left(|A|^{1-t} |A^*|^{1-t} \right) &\leq& w^{1/2}\left(|A|^{1-t} |A^*|^{1-t} \right)\\
		&\leq& \left\||A|^{1-t} |A^*|^{1-t} \right\|^{1/2}\\
		&\leq & \left\||A| |A^*|^{} \right\|^{(1-t)/2}
		= \left\|A^2 \right\|^{(1-t)/2}.
	\end{eqnarray*}
	Therefore, it follows from Theorem \ref{th6} that, for all $t\in [0,1]$,
	\begin{eqnarray*}
		w(A) & \leq & \frac12 \|A\|+ \frac12 \|A\|^t r^{1/2}\left(|A|^{1-t} |A^*|^{1-t} \right)\\
		& \leq & \frac12 \|A\|+ \frac12 \|A\|^t w^{1/2}\left(|A|^{1-t} |A^*|^{1-t} \right)\\
			& \leq & \frac12 \|A\|+ \frac12 \|A\|^t \left \| |A|^{1-t} |A^*|^{1-t} \right\|^{1/2}\\
		&\leq& \frac12 \|A\|+ \frac12 \|A\|^t \left\|A^2 \right\|^{(1-t)/2}.
	\end{eqnarray*}
	In particular, considering $t=0$, we get
	\begin{eqnarray*}
		w(A) &\leq&  \frac12 \|A\|+ \frac12  r^{1/2}\left(|A|^{} |A^*|^{} \right) \\ &\leq & \frac12 \|A\| +\frac12 \|A^2\|^{1/2},
	\end{eqnarray*}
which was also proved in \cite[Theorem 2.1 and Remark 2.2]{Bhunia_AM_2021} by using different approach.
\end{remark}

The next lemma that is needed for our purpose is as follows.

\begin{lemma}\label{lem8} \cite[Corollary 2.17]{Bhunia_BSM_2021}
	Let $B,C\in \mathcal{B}(\mathcal{H})$. Then
	$$ w^r(BC) \leq \frac1{2} w^2\left(|C|^r+i |B^*|^r \right) ,$$ for all $r\geq 2.$
\end{lemma}

Using the above lemma we prove the following theorem.

\begin{theorem}\label{th7}
		If $A\in \mathcal{B}(\mathcal{H})$, then
	\begin{eqnarray*} 
		w^r(A) &\leq& \frac12 w^{2} \left(  |A|^{rt} +i|A^*|^{r(1-t)} \right) \\
		&\leq& \frac12  \left\|  |A|^{2rt} +|A^*|^{2r(1-t)} \right\|,
	\end{eqnarray*}
	for all $t\in [0,1]$ and for all $r\geq 2.$ In particular, for $r=2$
	\begin{eqnarray}\label{par7}
		w(A) &\leq& \frac1{\sqrt2} w^{} \left(  |A|^{2t} +i|A^*|^{2(1-t)} \right)\\
		&\leq& \frac1{\sqrt2} \left\|  |A|^{4t} +|A^*|^{4(1-t)} \right\|^{1/2}\notag.
	\end{eqnarray}
for all $t\in [0,1]$.
\end{theorem}
\begin{proof}
	Let $A=U|A|$ be the polar decomposition of $A$.	By considering $B=U|A|^{1-t}$ and $C=|A|^{t}$ in Lemma \ref{lem8}, we obtain the desired first inequality. The next inequalities follow easily.
\end{proof}

%\begin{remark} Let $A\in \mathcal{B}(\mathcal{H}).$ 
In particular, considering $t=\frac12 $ in the inequality \eqref{par7}, we get 
	\begin{eqnarray*}
		w^2(A) &\leq& \frac1{2} w^{2} \left(  |A|^{} +i|A^*|^{} \right)\\ &\leq & \frac1{2} \left\|  |A|^{2} +|A^*|^{2} \right\|,
	\end{eqnarray*}
	which was also proved in \cite[Corollary 2.15 and Remark 2.16]{Bhunia_BSM_2021} using different technique.
%\end{remark}
To prove our final result we need the following lemma.
\begin{lemma}\label{lem9}\cite[Corollary 2.11]{Bhunia_GMJ_2023}
	Let $B,C\in \mathcal{B}(\mathcal{H})$ be such that $|B|C=C^*|B|$. Then
	$$ w(BC) \leq \frac1{\sqrt{2}}r(C)w\left( |B|+i|B^*|\right) .$$
\end{lemma}

\begin{theorem}\label{th8}
	If $A\in \mathcal{B}(\mathcal{H})$, then
	\begin{eqnarray*} 
		w(A) &\leq& \frac1{\sqrt{2}} \|A\|^t w \left(  |A|^{1-t} +i|A^*|^{1-t} \right) \\
		&\leq&  \|A\|^t  \left\|  \frac{|A|^{2(1-t)} +|A^*|^{2(1-t)}}2 \right\|^{1/2},
	\end{eqnarray*}
	for all $t\in [0,1]$. In particular, for $t=\frac12$
	\begin{eqnarray}\label{par8}
		w(A) &\leq& \frac1{\sqrt{2}} \|A\|^{1/2} w \left(  |A|^{1/2} +i|A^*|^{1/2} \right) \\
		&\leq&  \|A\|^{1/2}  \left\|  \frac{|A|^{} +|A^*|^{}}{2} \right\|^{1/2}\notag.
	\end{eqnarray}
	
\end{theorem}
\begin{proof}
	Let $A=U|A|$ be the polar decomposition of $A$.	By considering $B=U|A|^{1-t}$ and $C=|A|^{t}$ in Lemma \ref{lem9}, we obtain the desired first inequality. The next inequalities follow easily.
\end{proof}

\begin{remark}
We would like to remark that the bound \eqref{par8} is sharper than the bound  
\begin{eqnarray}\label{bound}
	w(A)\leq  \|A\|^{1/2}  \left\|  \frac{|A|^{} +|A^*|^{}}{2} \right\|^{1/2}.
\end{eqnarray}
The bound \eqref{bound} follows from the bounds $w(A)\leq \|A\|$ (see in \eqref{eqv1}) and  $w(A)\leq \left\|  \frac{|A|^{} +|A^*|^{}}{2} \right\|$ (see in \cite{Kittaneh_2003}).
\end{remark}

%\noindent \bf{Declarations.}\\
%\noindent {\bf{Conflict of Interest.}} The author declare that there is no conflict of interest.

\bibliographystyle{amsplain}

\begin{thebibliography}{99}
	
	\bibitem{Abu_RMJM_2015} A. Abu-Omar, and F. Kittaneh, Upper and lower bounds for the numerical radius with an application to involution operators, Rocky Mountain J. Math. 45 (2015), no. 4, 1055--1065.
	
	\bibitem{Abu_STU_2013}  A. Abu-Omar, and F. Kittaneh, A numerical radius inequality involving the generalized Aluthge transform, Studia Math. 216 (2013), no. 1, 69--75.
	
	
	
	
\bibitem{Bag_MIA_2020} S. Bag, P. Bhunia,  and K. Paul, Bounds of numerical radius of bounded linear operator using $t$-Aluthge transform, Math. Inequal. Appl.  23 (2020), no. 3, 991--1004.	

\bibitem{Bhunia_ASM_2022} P. Bhunia, and K. Paul, Refinement of numerical radius inequalities of complex Hilbert space operators, Acta Sci. Math. (Szeged) (2023). https://doi.org/10.1007/s44146-023-00070-1

\bibitem{Bhu23} P. Bhunia, Numerical radius inequalities of bounded linear operators and ($\alpha,\beta$)-normal operators, (2023). 
https://doi.org/10.48550/arXiv.2301.03877

\bibitem{Bhunia_GMJ_2023} P. Bhunia, S. Jana, and K. Paul, Numerical radius inequalities and estimation of zeros of polynomials, (2023). 
https://doi.org/10.48550/arXiv.2301.03159
%Georgian Math. J. (2023), to appear.


\bibitem{book} P. Bhunia, S.S. Dragomir, M.S. Moslehian, and K. Paul, Lectures on numerical radius inequalities, Infosys Science Foundation Series, Infosys Science Foundation Series in Mathematical Sciences, {Springer Cham}, 2022, XII+ 209 pp. ISBN: 978-3-031-13670-2. https://doi.org/10.1007/978-3-031-13670-2
	
	\bibitem{Bhunia_LAMA_2022} P. Bhunia, and K. Paul, Some improvements of numerical radius inequalities of operators and operator matrices, Linear Multilinear Algebra 70 (2022), no. 10, 1995--2013.
	
\bibitem{Bhunia_LAA_2021} P. Bhunia, and K. Paul,  Development of inequalities and characterization of equality conditions for the numerical radius, Linear Algebra Appl. 630 (2021), 306--315. 
	
	\bibitem{Bhunia_AM_2021}  P. Bhunia, and K. Paul, Furtherance of numerical radius inequalities of Hilbert space operators, Arch. Math. (Basel) 117 (2021), no. 5, 537--546.
	
	\bibitem{Bhunia_RIM_2021} P. Bhunia, and K. Paul, Proper improvement of well-known numerical radius inequalities and their applications, Results Math. 76 (2021), no. 4, Paper No. 177, 12 pp. 
	
	\bibitem{Bhunia_BSM_2021} P. Bhunia, K. Paul, New upper bounds for the numerical radius of Hilbert space operators, Bull. Sci. Math. 167 (2021), Paper No. 102959, 11 pp. 
	
	\bibitem{Bhunia_LAA_2019}  P. Bhunia, S. Bag, and K. Paul, Numerical radius inequalities and its applications in estimation of zeros of polynomials, Linear Algebra Appl. 573 (2019), 166--177.
	
	
	\bibitem{Buzano} M.L. Buzano, Generalizzatione della disuguaglianza di Cauchy-Schwarz, Rend. Semin. Mat. Univ. Politech. Torino 31(1971/73) (1974) 405--409.
	
	\bibitem{D08} S.S. Dragomir,  Some inequalities for the norm and the numerical radius of linear operators in Hilbert spaces, Tamkang J. Math. 39 (2008), no. 1, 1--7.
	

\bibitem{El_STU_2007} M. El-Haddad and F. Kittaneh, Numerical radius inequalities for Hilbert space operators. II,   Studia Math. 182 (2007), no. 2, 133--140.
	
	
	\bibitem{Kittaneh_LAMA_2023} F. Kittaneh, H.R. Moradi, and M. Sababheh, Sharper bounds for the numerical radius, Linear Multilinear Algebra, (2023). 
	https://doi.org/10.1080/03081087.2023.2177248
	
	\bibitem{Kittaneh_STD_2005}  F. Kittaneh, Numerical radius inequalities for Hilbert space operators, Studia Math. 168 (2005), no. 1, 73--80.
	
	\bibitem{Kittaneh_2003} F. Kittaneh, Numerical radius inequality and an estimate for the numerical radius of the Frobenius companion matrix, Studia Math. 158 (2003), no. 1, 11--17.
	
	\bibitem{Kittaneh_JFA_1997} F. Kittaneh, Norm inequalities for certain operator sums, J. Funct. Anal. 143 (1997), 337--348.
	
	 
	\bibitem{Kit88} F. Kittaneh, Notes on some inequalities for Hilbert space operators, {Publ. Res. Inst. Math. Sci.} 24 (1988), no. 2, 283--293.
	
\bibitem{MC} C.A. McCarthy, $C_{p}$, Israel J. Math. 5 (1967), 249--271.
	
\bibitem{book2} P.Y. Wu and H-L Gau, Numerical ranges of Hilbert space
operators. Encyclopedia of Mathematics and its Applications, 179. Cambridge
University Press, Cambridge, 2021. xviii+483 pp. ISBN: 978-1-108-47906-6
47-02
	
\bibitem{Yam_STU_2007} T. Yamazaki,  On upper and lower bounds for the numerical radius and an equality condition, Studia Math. 178 (2007), no. 1, 83--89.


\end{thebibliography}

\end{document}